\documentclass{amsart}[12pt]
\usepackage{graphicx}
\usepackage{amssymb}
\usepackage{amsmath,amsthm}
\usepackage{hyperref}
\usepackage{amsthm}
\usepackage[english]{babel}
\usepackage{mathrsfs}

\newtheorem{main}{Theorem}

\newtheorem{mcor}[main]{Corollary}
\newtheorem{theorem}{Theorem}[section]
\newtheorem{thm}[theorem]{Theorem}

\newtheorem{prop}[theorem]{Proposition}

\newtheorem{cor}[theorem]{Corollary}
\theoremstyle{definition}
\newtheorem{definition}[theorem]{Definition}
\newtheorem{notation}[theorem]{Notation}
\newtheorem{example}[theorem]{Example}
\newtheorem{examples}[theorem]{Examples}

\newtheorem{remark}[theorem]{Remark}
\newtheorem{question}[theorem]{Question}

\def\ca{\curvearrowright}
\def\G{\Gamma}
\def\g{\gamma}
\def\ra{\rightarrow}

\def\la{\lambda}
\def\La{\Lambda}

\def\bb{\mathbb}

\def\ve{\varepsilon}
\def\vp{\varphi}
\def\de{\delta}

\def\Cal{\mathcal}

\def\OO{\mathcal O}
\def\HH{\mathcal H}

\DeclareMathOperator*{\Lim}{{Lim}}

\numberwithin{equation}{section}

\newcommand{\ip}[2]{\langle #1, #2 \rangle}
\newcommand{\abs}[1]{\lvert#1\rvert}

\providecommand{\nor}[1]{\lVert #1 \rVert}

\begin{document}

\title[Inner amenability and central sequences]{Inner amenability for groups and central sequences in factors}

\author[Chifan]{Ionut Chifan}
\address{Department of Mathematics, University of Iowa, 14 MacLean Hall, IA 52242, USA and IMAR, Bucharest, Romania} \email{ionut-chifan@uiowa.edu}
\thanks{I.C.\ was supported in part by NSF Grant \#1263982.}
\author[Sinclair]{Thomas Sinclair}
\address{Department of Mathematics, University of California, Los Angeles, Box 951555, Los Angeles, CA 90095-1555} \email{thomas.sinclair@math.ucla.edu}
\thanks{T.S.\ was supported by an NSF RTG Assistant Adjunct Professorship.}
\author[Udrea]{Bogdan Udrea}
\address{Department of Mathematics, University of Illinois, Urbana Champaign, IL, USA and IMAR, Bucharest, Romania} \email{budrea@illinois.edu}
\thanks{}

\subjclass[2010]{22D40; 20F65}

\date{\today}

\dedicatory{}

\keywords{quasi-cocycles, inner amenability, property \emph{Gamma} of Murray and von Neumann}

\begin{abstract} We show that a large class of i.c.c., countable, discrete groups satisfying a weak negative curvature condition are not inner amenable. By recent work of Hull and Osin, our result recovers that mapping class 
groups and ${\rm Out}(\bb F_n)$ are not inner amenable. We also show that the group-measure space constructions associated to free, strongly ergodic p.m.p.\ actions of such groups do not have property Gamma of Murray and von 
Neumann.

\end{abstract}

\maketitle

\tableofcontents

\newpage

\section*{Introduction} 

The study of central sequences has occupied a prominent place in the classification of $\rm II_1$ factors. In their seminal investigations Murray and von Neumann \cite{Murray-vN-4} defined a $\rm II_1$ factor $M$ to have 
\emph{property Gamma} if there exists a net of unitaries $(u_n)$ in $M$ with $\tau(u_n)\equiv 0$ and such that $\nor{x u_n - u_n x}_2\to 0$, for all $x\in M$. In particular, they were able to demonstrate that the free group 
factor $L(\bb F_2)$ does not have property Gamma, providing the first demonstration of the existence of two non-isomorphic, non-hyperfinite $\rm II_1$ factors. The study of property Gamma played an important role in the 
celebrated classification results of McDuff \cite{McD} and Connes \cite{Connes-inj}.

 In the 1970s Effros \cite{Effros} introduced an analog of property Gamma for discrete groups, which he termed \emph{inner amenability}. A discrete group $\G$ is called inner amenable if there exists a net 
 $\xi_n\in \ell^2(\G\setminus\{e\})$ of unit-norm vectors such that $\nor{u_\g\xi_n - \xi_n u_\g}_2\to 0$, for all $\g\in \G$. A trivial consequence of Theorem 2.1 in \cite{Connes-inj} is that, for i.c.c.\ discrete groups, 
 inner-amenability is a weaker property than the group von Neumann algebra having property Gamma. However, examples of inner amenable groups whose von Neumann algebras do not possess property Gamma have only been recently 
 constructed by Vaes \cite{Vae}.

 For $\rm II_1$ factors without property Gamma, strong classification results have become achievable in a large part  through the development of Popa's deformation/rigidity theory \cite{P, PoICM, PoFree}. As the theory developed, it was readily 
 noticed that fairly mild ``deformability'' and ``rigidity'' assumptions could be used to demonstrate the absence of property Gamma, cf.\ \cite{CS, Hou, OzSolid, OPII, PetL2, PoFree}. In parallel, it was noticed that modest 
 ``negative curvature'' assumptions on a discrete group could be used to show non-inner amenability \cite{Harpe-hyp, DGO}. The goal of this paper is to fully develop the connections between these results through deriving non-inner 
 amenability of large classes of countable, discrete groups through operator algebraic methods, specifically the theory of ``weak'' deformations developed in \cite{CS, CSU, Sin}.
 
\subsection*{Statement of results} This paper is a continuation in the series of papers \cite{CS, CSU} exploring the consequences of negative-curvature phenomena in geometric group theory to the structure of group and group-measure space factors. The first in the series \cite{CS} dealt with structural results in the context of the strongest type of negative curvature condition, namely Gromov hyperbolicity. The essential result obtained therein was the extension of the strong solidity results of Ozawa and Popa \cite{OPI, OPII} and of the second author \cite{Sin} from lattices in rank one semisimple Lie groups to Gromov hyperbolic groups in general. In the second paper in the series \cite{CSU}, this result was further refined in two ways: first, to cover all weakly amenable groups satisfying a weaker negative curvature condition, relative hyperbolicity; and second, to cover products of such groups. The starting point of all of these results is the conversion of the negative curvature condition into a natural 
cohomological-type condition (relative $\Cal{QH}_{reg}$/bi-exactness) which is used to construct a weak deformation of the von Neumann algebra to which generalized ``spectral gap rigidity'' arguments of the type developed by Popa \cite{PoFree} and Ozawa and Popa \cite{OPI, OPII} are applied to obtain the desired classification results. The results of \cite{CS, CSU} were subsequently extended to general crossed product factors of hyperbolic groups in \cite{PoVa-hyp}.

Recent progress in geometric group theory has been on obtaining structural results for groups satisfying the much weaker negative curvature condition of ``acylindrical hyperbolicity,'' cf.\ \cite{Osin-ac}, or various 
equivalent geometric properties such as the condition of admitting a hyperbolic, WPD element, cf.\ \cite{BBF}; admitting a proper, infinite hyperbolically embedded subgroup \cite{DGO}; and weak acylindricity 
\cite{Ham-jems}, among others. In particular, it was shown by Dahmani, Guirardel, and Osin \cite{DGO}, that any group satisfying one of these conditions is not inner amenable. In parallel with these advances in 
geometric group theory, we introduce a cohomological-type version of weak negative curvature which we will use to classify the structure of central sequences for the group von Neumann algebra and the related 
group-measure space constructions.

\begin{notation} Let $\G$ be a countable discrete group, and let $\Cal G$ be a family of subgroups of $\G$. In order to simplify the statements of the results, throughout the paper we will use the notation set 
forth here.  We will say that $\G$ satisfies the condition $\mathsf{NC}$ relative to the family $\Cal G$ (abbreviated $\mathsf{NC}(\Cal G)$) if $\G$ satisfies one of the following statements:

\begin{itemize}
\item $\G$ admits an unbounded quasi-cocycle into a non-amenable orthogonal representation which is mixing with respect to $\Cal G$;
\item $\G$ admits a symmetric array into a non-amenable orthogonal representation so that the array is proper with respect to $\Cal G$.
\end{itemize}

\noindent The group $\G$ satisfies condition $\mathsf{NC}$ if it satisfies condition $\mathsf{NC}$ relative to the family consisting of the trivial subgroup.
\end{notation}

\noindent The condition $\mathsf{NC}$ is satisfied for all groups in the class $\Cal D_{reg}$ of Thom \cite{Tho} as well as the class $\Cal{QH}$ of the first two authors \cite{CS}. 
We refer the reader to section \ref{sec-background} below for relevant terminology and examples. 

In this paper we obtain a complete classification of the asymptotic central sequences of arbitrary crossed products factors $M=A\rtimes \G$ associated with groups $\G$ satisfying condition $\mathsf{NC}(\Cal G)$. In more colloquial
terms, we will be showing that all sequences which asymptotically commute with the entire factor $M$ must asymptotically ``live'' close to the (canonical) von Neumann subalgebras of $M$ arising from the subgroups of $\mathcal G$. Basic
examples can be constructed to show that this control is actually sharp. In particular, this result provides large natural classes of examples of i.c.c.\ groups whose factors do not possess property \emph{Gamma} of Murray and von Neumann. 
We now state the results.

\begin{main}\label{control-central-seq} Let $\Gamma$ be a countable discrete group together with a family of subgroups $\mathcal G$.  Let $\Gamma \curvearrowright A$ be any trace preserving action on an amenable, finite von Neumann 
algebra and denote by $M = A \rtimes \Gamma$. Also assume that $\omega$ is a free ultrafilter on the positive integers $\mathbb N$.  

If $\G$ satisfies condition $\mathsf{NC}(\Cal G)$, then for any sequence $(x_n)_n\in M' \cap M^{\omega}$ there exists a finite subset $\mathcal F\subseteq \mathcal G$ such that 
$(x_n)_n\in \vee_{\Sigma \in \mathcal F}(A \rtimes \Sigma)^{\omega} \vee M$.\end{main}

\begin{mcor} If $\G$ is an i.c.c., countable, discrete group which admits a non-degenerate, hyperbolically embedded subgroup, then $L\G$ does not have property Gamma. If $\G\ca (X,\mu)$ is any strongly ergodic p.m.p.\ action, then the group-measure space factor $L^\infty(X)\rtimes \G$ does not have property Gamma. In particular, this applies to non-virtually abelian mapping class groups $\Cal{MCG}(\Sigma)$ for $\Sigma$ a (punctured) closed, orientable surface as well as ${\rm Out}(\bb F_n)$, $n\geq 3$.
\end{mcor}

In the group factor case, the result may be sharpened to further rule out inner amenability.  We mention in passing that while this result is group theoretical in nature, the proof is
 is operator algebraic, rooted in Popa's deformation/rigidity theory.

\begin{main}\label{noninner-amenable} Let $ \G$ be an i.c.c.\ group together with a family of subgroups $\mathcal G$. Assume that $\G$ is i.c.c.\ over every subgroup $\Sigma \in \mathcal G$ (cf.\ Definition \ref{def-icc}).  
If $\G$ satisfies condition $\mathsf{NC}(\Cal G)$, then $\G$ is non-inner amenable. \end{main}

Since a non-amenable group is known to have positive first $\ell^2$-Betti number if and only if it admits an unbounded $1$-cocycle into its left-regular representation \cite{PeTho}, Corollary 2.4, we have the following easy 
corollary. Surprisingly, to the best of our knowledge this is the first time this result has appeared in print, though we were informed by Taka Ozawa that he had previously obtained this result in unpublished work.

\begin{mcor}\label{betti} If $\G$ is an i.c.c.\ countable discrete group with positive first $\ell^2$-Betti number, then $\G$ is not inner amenable.
\end{mcor}

When combined with the main result of Hull and Osin from \cite{HO} our theorem also recovers the following earlier result due to Dahmani, Guirardel, and Osin.

\begin{mcor}[Dahmani, Guirardel, and Osin \cite{DGO}] If $\G$ is an i.c.c., countable, discrete group which admits a non-degenerate, hyperbolically embedded subgroup, then $\G$ is not inner amenable.
\end{mcor}

\noindent Additionally, for such groups the authors were also able to demonstrate in \cite{DGO} simplicity of the reduced $\rm C^*$-algebra $C^*_r(\G)$. We were not able to obtain any positive results for the more general 
class of groups satisfying condition $\mathsf{NC}$, though we remark on some possible connections between the results outlined here and $\rm C^*$-simplicity in section \ref{simplicity}.

Recall, a II$_1$ factor is said to be \emph{prime} if it is not isomorphic to a tensor product of diffuse factors.

\begin{question} If $\G$ is an exact, non-amenable, i.c.c., countable, discrete group which admits an unbounded quasi-cocycle into $\ell^2(\G)^{\oplus\infty}$, is $L\G$ prime? \end{question}

\noindent Note that Peterson showed (Corollary 4.6 in \cite{PetL2}) that primeness of $L\G$ does follow from the much more restrictive assumption that $\G$ admits an unbounded $1$-cocycle into $\ell^2(\G)^{\oplus\infty}$ 
(the assumption of exactness is not necessary). Primeness is also known when the quasi-cocycle is proper, even extending to the case of proper arrays, by Theorem A in \cite{CS}, though this is implicitly due to Ozawa 
(Theorem 1 in \cite{OzSolid} via Remark 1.10 in \cite{CS}).  An affirmative answer to the question would be sharp: by Proposition 1.4 in \cite{CS} the group $\bb F_2\times \bb F_2$, for instance, admits an unbounded 
(but not proper) array into its left-regular representation.


\subsection*{Outline of the paper} The first section contains the necessary background material, definitions, and examples.  The second section consists of the statement and proofs of the main new technical results 
on quasi-cocycles. The proofs of the main results stated above, as well as other applications of the technique, form the third and last section of the paper.


\section{Background and Methods}\label{sec-background}

\subsection{Arrays and quasi-cocycles} Arrays were introduced by the first two authors in \cite{CS} as a language for unifying the concepts of length functions and $1$-cocycles into orthogonal representations. 
In practice arrays can be used either to strengthen the concept of a length function by introducing a representation or to introduce some ``geometric'' flexibility to the concept of a $1$-cocycle. See section 
1 of \cite{CS} for an in-depth discussion of this concept and its relation with ``negative curvature'' in geometric group theory).

\begin{definition} Assume that $\G$ is a countable, discrete group together with  $\mathcal G=\{\Sigma_i\,:\,i\in I\}$, a family of subgroups  of $\G$ and $\pi : \G\ra \OO(\mathcal H)$, an orthogonal representation. 
Following Definition 1.4 from \cite{CS}, we say that the group $\G$ admits an \emph{array} into $\mathcal H$ if there exists a map $r: \G\ra \HH$ which satisfies the following \emph{bounded equivariance} condition:

\begin{equation*}\sup_{\de\in \G}\|r(\g \de)-\pi_\g(r(\de)) \|=C(\g)<\infty, \text{ for all }\g\in \G.
\end{equation*}

\noindent An array $r$ is said to be \emph{symmetric} \{\emph{anti-symmetric}\} if we have that \[\pi_\g(r(\g^{-1}))=r(\g)\ \{\pi_\g(r(\g^{-1}))= -r(\g)\},\] for all $\g\in\G$. It is \emph{proper relative to $\mathcal G$} 
if for every $C>0$ there are finite subsets $F\subset G$ and $\Cal K\subset \Cal G$ such that \[ B_C := \{\g\in\G : \nor{r(\g)}\leq C\} \subset \bigcup_{K\in\Cal K} FKF.\]

\end{definition}

\noindent For a detailed list of properties of groups that admit nontrivial arrays the reader may consult  \cite{CS, CSU}. 

 The main source of examples of arrays on groups are quasi-cocycles. As before let $\G$ be a countable group together with  $\mathcal G=\{\Sigma_i\,:\,i\in I\}$ a family of subgroups  of $\G$ and let 
 $\pi : \G\ra \OO(\mathcal H)$ be an orthogonal representation. 
 
 \begin{definition} A map $q:\G\ra \mathcal H$ is called a quasi-cocycle if there exists a constant $D\geq 0$  such that
\begin{equation}\label{quasicocyclerel1}\|q(\g \la)-\pi_\g(q(\la)-q(\g)) \|\leq D, \text{ for all }\g,\la\in \G.
\end{equation} 

\noindent The infimum over all such $D$ is denoted by $D(q)$ and is called the \emph{defect} of $q$. When the defect is zero $q$ is actually a $1$-cocycle with coefficients in $\pi$ (i.e. an element in $Z^1(\G,\pi)$).  
Any bounded map $b:\G \ra \mathcal H$ is automatically a quasi-cocycle whose error does not exceed three times the uniform bound of $b$.  

\end{definition}

 It was observed by Thom \cite{Tho} that any quasi-cocycle lies within bounded distance from an anti-symmetric one. We denote the space of anti-symmetric quasi-cocycles associated to the representation 
 $\pi$ as $QZ^1_{as}(\G,\pi)$ and the subspace of those which are bounded as $QB^1_{as}(\G,\pi)$. The first quasi-cohomology space is then defined to be $QH^1_{as}(\G,\pi) := QZ^1_{as}(\G,\pi)/QB^1_{as}(\G,\pi)$. In particular, 
 if $\pi$ is the left-regular representation $\lambda_\G$, then $QH^1_{as}(\G,\lambda_\G)$ has the structure of a right $L\G$-module.

\begin{definition} A group $\G$ is said to be in the class $\Cal D_{reg}$ if $\dim_{L\G} QH^1_{as}(\G,\lambda_{\G})\not= 0$.
\end{definition}

\noindent By Lemma 2.8 in \cite{Tho} and Corollary 2.4 in \cite{PeTho} it is observed that a countable, discrete group $\G$ is in the class $\Cal D_{reg}$ if and only if it admits an unbounded quasi-cocycle into its left-regular 
representation.

\subsection{Groups satisfying condition $\mathsf{NC}$}

We will now describe some specific examples and constructions of classes of groups satisfying condition $\mathsf{NC}$ with respect to some explicit families of subgroups. While all groups which are (relatively) bi-exact belong to 
this class, stronger results, cf.\ \cite{BrOz, CS, CSU, OP}, are known in this case, so we will focus our attention here on weaker ``negative curvature'' conditions which can be used to construct unbounded quasi-cocycles. Before 
doing so, in order to make the exposition more self-contained and to furnish some more familiar examples, we will recall for classes of groups arising  from canonical groups constructions like (semi-)direct products, (amalgamated)
free products, and HNN-extensions, how to go about constructing quasi-cocycles algebraically. For more details we refer to section 2.2 in \cite{CSU}.

\begin{examples} Each group in the following classes satisfies condition $\mathsf{NC}$ with respect to the associated family of subgroups $\Cal G$.

\begin{enumerate}

\renewcommand{\labelenumi}{{\bf \alph{enumi}.}}

\item If $\Sigma< \G_1,\G_2$ are groups and $\G :=\G_1\star_\Sigma\G_2$ is the corresponding amalgamated free product then it is well known that $H^1(\G ,\ell^2(\G/\Sigma))\neq \{0\}$. Let $\Cal G := \{\Sigma\}$.

\item If $\Sigma <\G$ are groups, $\theta:\Sigma \ra \G$ is a monomorphism, then, denoting by  $\G := {\rm HNN}(\G,\Sigma,\theta)$ the corresponding HNN-extension, we have again $H^1(\G ,\ell^2(\G/\Sigma))\neq \{0\}$ and 
$\Cal G := \{\Sigma\}$.

\item Let $H$, $\G$ be countably infinite discrete groups, let $I$ be a $\G$-set, and consider the generalized wreath product group $ H\wr_I\G := (\bigoplus_I H) \rtimes \G$. For every $i \in I$, denote by 
$\G_i := \lbrace \gamma \in \G : \gamma i = i \rbrace$ the stabilizer subgroup of $i$ and let $\Cal G := \{\G_i : i\in I\}$. If either $\G_i$ is not co-amenable for every $i$ or $H$ is non-amenable, then the group 
$  H\wr_I\G$ admits a 1-cocycle into the representation $ \ell^2((  H\wr_I\G) / \G)$, which under the assumptions is easily seen to be non-amenable and mixing with respect to the family of stabilizers $\Cal G$.

\item Let $\G$ be a group that acts on a tree $\mathcal T =(\mathsf V, \mathsf E)$. We assume in addition that for each edge $e\in \mathsf E$ its stabilizer group $\G_e:=\{\g\in \G\,:\,\g e=e \}$ is not co-amenable. 
Since $\G$ acts on a tree, there exists a 1-cocycle into the semi-regular orthogonal representation  $\la_{\mathsf E} :=\bigoplus_{e\in\mathsf E} \ell^2(\G/\G_{e})$ where the group acts by left translation on each summand. 
From the assumptions it is clear that $\la_{\mathsf E}$ is non-amenable and mixing with respect to the family $\mathcal G:=\{\G_e : e\in\mathsf E\}$.

\end{enumerate}
\end{examples}

\noindent Another class of examples comes from lattices in locally compact, second countable (l.c.s.c.) groups.

\begin{example}\label{ex-lattice} Consider a l.c.s.c.\ group $G= G_1\times G_2$, where $G_1$ has property (HH) of Ozawa and Popa \cite{OPII}, i.e., $G_1$ admits a proper cocycle into a non-amenable, mixing orthogonal 
representation. If $\G < G$ is a lattice then $\G$ satisfies condition $\mathsf{NC}(\Cal K)$ for $\Cal K$ the set of subgroups $K < \G$ such that the projection $pr_1(K)$ of $K$ into $G_1$ is pre-compact. 
(Note the lattice assumption is only necessary to ensure non-amenability of the restricted representation.)
\end{example}

We will now describe some recent, innovative methods for building quasi-cocycles through the use of geometric methods in group theory. Some of the first results in this direction come from the seminal work of Mineyev 
\cite{Min} and Mineyev, Monod, and Shalom \cite{MMS}, who showed that if $\G$ is a Gromov hyperbolic group, then $\G$ admits a proper quasi-cocycle into a finite multiple of its left-regular representation, hence belongs to 
$\Cal D_{reg}$. This work was generalized to groups which have relative hyperbolicity to a family of subgroups by Mineyev and Yaman \cite{MinYa}. Hamenst\"adt \cite{Ham-jems} showed that all weakly acylindrical 
groups -- in particular, non-virtually abelian mapping class groups and ${\rm Out}(\bb F_n)$, $n\geq 2$ -- belong to the class $\Cal D_{reg}$. Recently a unified approach to these results was developed by Hull and Osin \cite{HO} 
and independently by Bestvina, Bromberg, and Fujiwara \cite{BBF} to show that for every group $\G$ which admits a non-degenerate, hyperbolically embedded subgroup belongs to the class $\Cal D_{reg}$ via an extension theorem on 
quasi-cohomology. In fact, by very recent work of Osin \cite{Osin-ac} the weak curvature conditions used in both papers, as well as Hamenst\"adt's weak acylindricity condition, are equivalent to the notion of ``acylindrical 
hyperbolicity'' formulated by Bowditch, cf.\ [op.\ cit.].
 
\begin{examples}\label{big-list} Collecting these results together, the following families of groups are known to be acylindrically hyperbolic. In particular they belong to the class $\Cal D_{reg}$, hence satisfy condition $\mathsf{NC}$:

\begin{enumerate}

\renewcommand{\labelenumi}{{\bf \alph{enumi}.}}

\item Gromov hyperbolic groups \cite{Min,MMS};
\item Groups which are hyperbolic relative to a family of subgroups as in \cite{MinYa};
\item The mapping class group $\Cal{MCG}(\Sigma)$ for any (punctured) closed, orientable surface $\Sigma$, provided that it is not virtually abelian \cite{Ham-jems};
\item ${\rm Out}(\bb F_n)$, $n\geq 2$ \cite{Ham-jems};
\item Groups which admit a proper isometric action on a proper ${\rm CAT}(0)$ space \cite{Sisto}.
\end{enumerate}
\end{examples}

We remark that it is unclear whether condition $\mathsf{NC}$ is closed under finite direct sums, though the following partial stability result is easily observed, cf.\ Proposition 1.7 in \cite{CS}.

\begin{prop}\label{nc-product} Let $\G_1$ and $\G_2$ be countable, discrete groups, and let $\Cal G_1$ and $\Cal G_2$ be respective families of subgroups. Consider the direct product $\G = \G_1\times \G_2$ equipped with the family of subgroups $\Cal G := \{\Sigma_1\times \G_2 : \Sigma_1\in\Cal G_1\}\cup\{\G_1\times\Sigma_2 : \Sigma_2\in\Cal G_2\}$. If both $\G_1$ and $\G_2$ either admit a symmetric array into a non-amenable representation which is proper with respect to $\Cal G_i$ or admit an unbounded quasi-cocycle into a non-amenable representation which is mixing with respect to $\Cal G_i$, $i=1,2$, then the same holds for $\G$ with respect to $\Cal G$.
\end{prop}

\noindent On the other hand, it is known that the classes of non-Gamma factors and non-inner amenable groups are closed under, respectively, finite tensor products (a non-trivial result of Connes \cite{Connes-inj}) and finite direct sums (folklore\footnote{Precisely, this follows from the observation that $\G$ is not inner amenable if and only if the action of $\G$ by conjugation on $\ell^2(\G\setminus\{e\})$ is non-amenable.}). This suggests that the condition $\mathsf{NC}$ as formulated may not be an optimal condition for establishing results along the lines of those stated in the introduction.

\subsection{A family of deformations arising from arrays}\label{sec:deformations}

We briefly recall from \cite{CS} the construction of a deformation from an array based on the main construction in \cite{Sin}. To do this we need to first recall the construction of the Gaussian action associated to a 
representation (see for example \cite{PeSi}). Let $\pi:\G \rightarrow \mathcal{O}(\mathcal H)$ be an orthogonal representation on a real Hilbert space, and consider the abelian von Neumann algebra $(D, \tau)$ generated by 
a family of unitaries $\omega(\xi),\xi \in \mathcal H$, subject to the following relations:

\begin{enumerate}
\item $\omega(\xi_1)\omega(\xi_2)=\omega(\xi_1+\xi_2)$, for any $\xi_1, \xi_2 \in \mathcal H$;
\item $\omega(-\xi)=\omega(\xi)^*$, for any $\xi \in \mathcal H$;
\item $\tau(\omega(\xi))=\exp(-\|\xi \|^2)$, for any $\xi \in \mathcal H$.
\end{enumerate}

The \emph{Gaussian action} of $\G$ on $(D, \tau)$ is defined by $\tilde{\pi}_\g(\omega(\xi))=\omega(\pi_\g(\xi))$, for all $\g \in \G$ and $\xi \in \mathcal H$. 
 \begin{definition}Let $\G \ca^\sigma(A, \tau)$ be a trace-preserving action of $\G$ on a finite von Neumann algebra $A$ and let $M=A \rtimes_{\sigma} \G$ be the cross-product von Neumann algebra. The \emph{Gaussian dilation} 
 associated to $M$ is the von Neumann algebra $\tilde{M}=(A\bar{\otimes}D)\rtimes_{\sigma \otimes \tilde{\pi}} \G$.
\end{definition}

 Let $q:\G \rightarrow \mathcal H$ be an array for the representation $\pi$ as above. The deformation is constructed as follows.  For each $t\in \mathbb{R}$, define the unitary
 $V_t\in \mathcal U(L^2(A)\otimes L^2(D)\otimes \ell^2(\G))$ by \[V_t(a \otimes d \otimes \delta_\g) := a \otimes \omega(tq(\g))d \otimes \delta_\g,\] for all $a \in L^2(A), d \in L^2(D)$, and $\g \in \G$. In \cite{CS} it was 
 proved that $V_t$ is a strongly continuous one parameter group of unitaries having the following transversality property:

\begin{prop}[Lemma 2.8 in \cite{CS}]\label{transversality}For each $t$ and any $\xi \in L^2(M)$, we have 
\begin{equation}2\|V_t(\xi)-e \cdot V_t(\xi)\|_{2}^2 \geq \|\xi-V_t(\xi)\|_{2}^2, \end{equation}
where $e$ denotes the orthogonal projection of $L^2(\tilde{M})$ onto $L^2(M)$.
 \end{prop}

The following ``asymptotic bimodularity'' property of the deformation $V_t$ is the most crucial consequence of the array property. It follows easily from Lemma 2.6 in \cite{CS}.

\begin{prop}\label{almostbimodular}Let $\G$ be a group, let $\pi:\G\ra \mathcal H$ be a unitary representation and let  $q:\G\ra \HH$ be an array for $\pi$.  Assume that $\G \ca^\sigma A$ is a trace preserving action and let $\G \ca^{\sigma \otimes \pi}A \bar{\otimes} D$ be the Gaussian construction associated to $\pi$.   Denote by $M=A\rtimes \G$ and $\tilde M=A\bar{\otimes} D\rtimes_{\sigma \otimes \tilde{\pi}} \G$ the corresponding crossed product algebras and notice that $M\subset \tilde M$. To $q$ we associate the path of isometries $V_t: L^2(M)\ra L^2(\tilde M)$ obtained by restriction from the $V_t$'s constructed above. Then for every $x,y \in A \rtimes_{\sigma, r} \G$ (the reduced crossed product) we have that 
\begin{equation}\lim_{t\ra 0}\sup_{\|\xi\|_2 \leq 1} \|xV_t(\xi)y-V_t(x\xi y)\|_2=0. \end{equation}  
\end{prop}

Let $\rho : \G \rightarrow \mathcal O(\mathcal H)$ be a orthogonal representation of $\G$ which admits no non-zero invariant vectors, i.e., $\rho$ is ergodic. In this case $\rho$ has \emph{spectral gap} if and only if it does not weakly contain the trivial representation. The orthogonal representation $\rho$ is \emph{non-amenable} if $\rho \otimes \rho$ is ergodic and has spectral gap.

\begin{prop}[Proposition 2.7 in \cite{PeSi}]\label{spectralgap} The orthogonal representation $\rho$ is non-amenable if and only if the Koopman representation $\tilde{\rho} : \G \rightarrow \mathcal U( L^2(D \ominus \mathbb{C}1))$ associated to the Gaussian action is ergodic and has spectral gap.
\end{prop}

\section{Technical Results}

In this section we present the new technique for working with quasi-cocycles which will allow us to prove the main results of the paper. Studying the properties of groups through various Hilbert space embeddings, e.g.,  
quasi-cocycles, which are compatible with some representation has emerged as a fairly important tool which captures many interesting aspects regarding the internal (algebraic) structure of the group. The principle is parallel 
to the use of geometric techniques (via a word-length metric) to deduce algebraic structure -- the difference being that while geometric techniques focus on deducing structure from assumptions on the length function itself, 
quasi-cocycle (or array) techniques impose structure from the existence of a generic Hilbert space-valued ``length'' function which is compatible with a specific representation.

Continuing this trend we show next that if the representation is mixing then the finite radius balls with respect to the natural length function induced by the quasi-cocycle are highly ``malnormal.''  As a consequence 
we show that we have large sets in the group which are asymptotically free. 

\begin{prop}\label{mixingmalnormal}Let $\,\G$ be a countable group together with a family of subgroups $\mathcal G$, let $\pi:\G\ra \mathcal O(\mathcal H)$ be an orthogonal representation that is mixing with respect to 
$\mathcal G$. Assume $q:\G\ra \mathcal H$ is a quasi-cocycle and for every $C\geq 0$ we denote by \[B _C :=\{\g\in \G :  \|q(\g)\|\leq C \}. \] Fix $C\geq 0$ and an integer $\ell\geq 2$. Also let $k_1,k_2,\ldots, k_\ell \in \G$ 
be elements such that, for each $1\leq i\leq \ell$ there exists an infinite sequence $(\gamma_{n,i})_n\in B_C$ such that for all $n$ we have $k_1 \gamma_{n,1}k_2\gamma_{n,2 }k_3 \gamma_{n,3}\cdots k_\ell\gamma_{n,\ell}=e$ and 
each of the following sets $\{\gamma_{n,1} :  n\in \mathbb N \}$, $\{\gamma_{n,1} k_2\gamma_{n,2}\,: \, n\in \mathbb N \}$, $\{\gamma_{n,1}k_2\gamma_{n,2} k_3\gamma_{n,3}\,: \, n\in \mathbb N\}, \ldots, \{\gamma_{n,1}k_2\cdots 
\gamma_{n,\ell-2} k_{\ell-1} \gamma_{n,\ell-1} : n\in \mathbb N\}$ converges to infinity with respect to $\mathcal G$. Then 
$\,k_1\in B_{(\ell-1)(2C+2D)}$. 
 
\end{prop}

\begin{proof}First we show that for all $x_1,x_2,\ldots, x_m\in B_C$, $b_2,b_3,\ldots, b_m\in \G$, and $k\in \Cal H$ we have 
\begin{equation}\label{2525}|\langle q(x_1b_2x_2\cdots x_{m-1}b_mx_m), k\rangle| \leq  (m-1)(2C+2D)\|k\|+ \sum^m_{i=2} |\langle\pi_{x_1b_2x_2\cdots b_{i-1}x_{i-1}}(q(b_i)),k\rangle|.
\end{equation}   
  
  To see this we argue by induction on $m$. When $m=2$, using the quasi-cocycle relation (\ref{quasicocyclerel1}) and the Cauchy-Schwarz inequality we have that  \begin{equation*}\begin{split} 
|\langle q(x_1b_2x_2), k\rangle| &\leq D\|k\|+ |\langle q(x_1)+\pi_{x_1}(q(b_2x_2)) ,k\rangle |\\
 &\leq (D +\|q(x_1)\|)\|k\|+ |\langle \pi_{x_1}(q(b_2x_2)) ,k\rangle |\\ 
&\leq (C + 2D)\|k\|+ |\langle \pi_{x_1}(q(b_2)) ,k\rangle |+ |\langle \pi_{x_1b_2}(q(x_2)) ,k\rangle |\\
&\leq (2C + 2D)\nor{k} + |\langle \pi_{x_1}(q(b_2)) ,k\rangle |.
\end{split}\end{equation*}
 
 \noindent For the inductive step assume (\ref{2525}) holds for $2,\dotsc,m$; thus, we have that 

\begin{multline}\label{2527} |\langle q(x_1b_2\cdots b_{m}x_{m}b_{m+1}x_{m+1} ), k\rangle| \leq |\langle \pi_{x_1b_2}q(x_2b_3\cdots b_{m}x_{m}b_{m+1}x_{m+1} ), k\rangle| + (C + 2D)\nor{k}+
|\langle \pi_{x_1}(q(b_2)),k\rangle |\\
\leq  m(2C+2D)\|k\|+  \sum^m_{i=2} |\langle\pi_{x_1b_2x_2\cdots b_{i-1}x_{i-1}}(q(b_i)),k\rangle|,
\end{multline}   
and we are done with the claim. $\blacksquare$\\

To prove our statement note that since $k^{-1}_1=\g_{n,1}k_2\g_{n,2} k_3\g_{n,3} \cdots k_\ell\g_{n,\ell} $, then (\ref{2525}) implies \begin{multline}\label{1'}
\|q(k^{-1}_1)\|^2 =\langle q(\g_{n,1}k_2\g_{n,2} k_3\g_{n,3} \cdots k_\ell\g_{n,\ell} ),q(k^{-1}_1)\rangle \\
\leq  (\ell-1)(2C+2D)\|q(k^{-1}_1)\|+\sum^\ell_{i=2} |\langle\pi_{\g_{n,1}k_2\g_{n,2}\cdots k_{i-1}\g_{n,i-1}}(q(k_i)),q(k^{-1}_1)\rangle|.
 \end{multline}
 
\noindent Since each of the sets $\{\g_{n,1} : n\in \mathbb N \}$, $\{\g_{n,1}k_2\g_{n,2}: n\in \mathbb N\}, \ldots,  \{\g_{n,1}k_{2}\g_{n,2}\cdots k_{\ell-1} \g_{n,\ell-1}: n\in \mathbb N\}$ converges to infinity relative to $\mathcal G$ and $\pi$ is a mixing relative to $\mathcal G$, taking the limit as $n\ra \infty$ in (\ref{1'}) we get  \[\|q(k^{-1}_1)\|^2 \leq  (\ell-1)(2C+ 2D)\|q(k^{-1}_1)\|.\]  Hence, by anti-symmetry we have \[\nor{q(k_1)} =\|q(k^{-1}_1)\| \leq    (\ell-1)(2C+2D).\]

\end{proof}

The next result is the main application of Proposition \ref{mixingmalnormal}.

\begin{cor}\label{quasi-control} Under the same assumptions and notations as in Proposition \ref{mixingmalnormal}, for every finite set $\,F\subset \G\setminus B_{2C+2D}$ there exists a subset $\,K\subset B_C$ which is small with 
respect to the family $\mathcal G$ such that \[\,F \left (B_C\setminus K\right ) \cap  \left (B_C\setminus K\right ) F=\emptyset.\]
Moreover, if the quasi-cocycle $q$ is bounded on each group in $\mathcal G$, then for every finite set $\,F\subset \G\setminus B_{6C+6D}$ one can find a subset $\,K\subset B_C$ which is small 
with respect to the family $\mathcal G$ and satisfies \[\,F \left (B_C\setminus  (K^2\cup K)\right ) F\left (B_C\setminus  (K^2\cup K)\right )\cap  \left (B_C\setminus (K^2\cup K)\right ) F\left (B_C\setminus  (K^2\cup K)\right ) F=\emptyset.\]
\end{cor}

\begin{proof} Let $k_1, k_2\in F\subset \Gamma \setminus B_{2C+2D}$ be fixed elements and let $(\gamma_{n,1})_n$, $(\gamma_{n,2})_n$ be sequences in $B_C$ such that 
$k_1B_C\cap B_C k_2=\{k_1\gamma_{n,1}=\gamma_{n,2}k_2 \,:\, n\in \mathbb N\}$. In particular we have $k_1\gamma_{n,1}k^{-1}_2\gamma^{-1}_{n,2}=e$ for all $n$ and by the first part the set 
$\{\gamma^{-1}_{n,1}\,:\,n\in \mathbb N \}$ is small over $\mathcal G$ and so is $K_{k_1,k_2}=\{\gamma_{n,1}\,:\,n\in \mathbb N \}$. This entails that $k_1(B_C\setminus K_{k_1,k_2})\cap (B_C\setminus K_{k_1,k_2}) k_2=\emptyset$ 
and hence if we let $K=\cup_{k_1,k_2\in F}K_{k_1,k_2}$ we see that $K$ is small with respect to $\mathcal G$ and $F(B_C\setminus K)\cap (B_C\setminus K)F =\emptyset$. 

For the second part let $k_1, k_2, k_3, k_4\in F\subset \Gamma \setminus B_{6C+6D}$ and assume that there exist infinite sequences $(\gamma_{n,i})_n\in B_C$ with $1\leq i\leq 4$ such that $k_1B_Ck_2B_C\cap B_Ck_3B_Ck_4:= \{k_1\g_{n,1}k_2\g_{n,2}=\g_{n,3}k_3\g_{n,4}k_4 \,:\,n\in \mathbb N\}$.  
This further implies that $k_1\g_{n,1}k_2\g_{n,2}k_4^{-1}\g_{n,4}^{-1}k_3^{-1}\g_{n,3}=e$, for all $n\in \mathbb N$. Then from the previous proposition it follows that at least one of the sets 
$\{\g_{n,1}\,:\,n\in \mathbb N\}$, $\{\g_{n,1}k_2\g_{n,2}\,:\,n\in \mathbb N\}$, or $\{\g_{n,1}k_2\g_{n,2}k^{-1}_4\g^{-1}_{n,4}\,:\,n\in \mathbb N\}$ must be small with respect to $\mathcal G$. 
If we have either the set  $\{\g_{n,1}\,:\,n\in \mathbb N\}$ or the set $\{\g_{n,1}k_2\g_{n,2}k^{-1}_4\g^{-1}_{n,4}\,:\,n\in \mathbb N\}=\{k_1^{-1}\g^{-1}_{n,3}k_3\,:\,n\in \mathbb N\}$ (and hence $\{\g_{n,3}\,:\,n\in \mathbb N\}$!) 
is small with respect to $\mathcal G$ then the conclusion follows imediately. So it remains to analyze the case when $\{\g_{n,1}\,:\,n\in \mathbb N\}$ converges to infinity with respect to $\Cal G$  while the set 
$\{\g_{n,1}k_2\g_{n,2}\,:\,n\in \mathbb N\}$ is small with respect to $\mathcal G$. Here we only need to investigate the case when there exist: an infinite sequence of positive integers $(r_n)_n$ such that 
$\{\g_{r_n,1}\,:\,n\in \mathbb N\}$ converges to infinity with respect to 
$\Cal G$; elements $m_1, m_2\in \G$ and a sequence of elements $a_n\in \Sigma$ for some $\Sigma \in \mathcal G$ such that $\g_{r_n,1} k_2\g_{r_n,2}=m_1 a_n m_2$, for all $n\in \mathbb N$. Note that the latter equation can be 
rewritten as  \begin{equation}\label{71} k_2\g_{r_n,2}m^{-1}_2a_n^{-1}m_1^{-1}\g_{r_n,1}=e,\text{ for all }n\in \mathbb N.\end{equation} If the set $\{\g_{r_n,2} \,:\,n\in \mathbb N\}$ would converge to infinity with respect to 
$\Cal G$ then, since so is $\{k_2\g_{r_n,2}m^{-1}_2a_n^{-1}\,:\,n\in \mathbb N\}=\{\g^{-1}_{r_n,1}m_1\,:\,n\in \mathbb N\}$, we would have by the previous proposition that $k_2\in B_{4C+4D}$ which is a contradiction. 
Thus $\{\g_{r_n,2} \,:\,n\in \mathbb N\}$ is small with respect to $\Cal G$ and from equation (\ref{71}) it follows that there exists a set $K$ which is small with respect to $\Cal G$ such that 
$\{\g_{r_n,1}\,:\,n\in \mathbb N\}\subset K^2$; this gives the desired conclusion.  
\end{proof}

In the case of mixing representations we get the following sharper result.

\begin{cor}\label{mixingmalnormal1}Let $\,\G$ be a countable, discrete group, and let $\pi:\G\ra \mathcal O(\mathcal H)$ be a mixing orthogonal representation. Assume $q:\G\ra \mathcal H$ is a quasi-cocycle and for every 
$C\geq 0$ we denote by $B _C :=\{\g\in \G \, : \, \|q(\g)\|\leq C \}$.  Let $\,C\geq 0$ and let  $\,k_1,k_2,\ldots, k_\ell \in \G$ be elements such that, for each $1\leq i\leq \ell$ there exists an infinite sequence 
$(\gamma_{n,i})_n\in B_C$ (without repetitions) such that we have $k_1 \gamma_{n,1}k_2\gamma_{n,2 }k_3 \gamma_{n,3}\cdots k_\ell\gamma_{n,\ell}=e$, for all $n$. Then there exists $1\leq j\leq n$ such that $\,k_j\in B_{\ell(C+D)}$. 
In other words, for every finite set $\,F\subset \G\setminus B_{2\ell(C+D)}$ there exists a finite subset $\,K\subset B_C$ such that 
\begin{equation}\label{72} e\notin [F \left (B_C\setminus K\right )]^\ell:=\underbrace{[F \left (B_C\setminus K\right )][F \left (B_C\setminus K\right )]\cdots [F \left (B_C\setminus K\right )]}_{\ell \text{-times}}.\end{equation}
In particular,  for every finite set $\,F\subset \G\setminus B_{2\ell(C+D)}$ and every $1\leq \kappa \leq \ell$ there exists a finite subset $\,K\subset B_C$ such that  \begin{equation*}[F \left (B_C\setminus K\right )]^\kappa \cap  [\left (B_C\setminus K\right ) F]^{\ell-\kappa}=\emptyset.\end{equation*}
\end{cor}

\begin{proof}First we claim that there exist two positive integers $1\leq s< t\leq \ell$, an infinite sequence $(r_n)_n$ of positive integers, and an element $k'_t \in \G$  such that  for all $n$ we have 
\begin{eqnarray} \label{inf1}&& \left |\{k_s\g_{r_n,s} \,:\,n\in \mathbb N \}\right |=\left |\{k_s\g_{r_n,s}k_{s+1}\g_{r_n,s+1} \,:\,n\in \mathbb N \}\right |=\\ 
\nonumber&&=\cdots =\left |\{k_s\g_{r_n,s}k_{s+1}\g_{r_n,s+1}\cdots k_{t-1}\g_{r_n, t-1} \,:\,n\in \mathbb N \}\right |=\infty;\\
\label{rel1}&&k_s\g_{r_n,s}k_{s+1}\g_{r_n,s+1}\cdots k_{t-1}\g_{r_n, t-1}k'_t \g_{r_n,t}=e.\end{eqnarray} 
Next we observe that if this is the case then applying Proposition \ref{mixingmalnormal} above it follows that $k_s\in B_{2(t-s)(C+D)}\subseteq B_{2\ell(C+D)}$. The remaining part of the statement follows easily from this.

So it only remains to show  (\ref{inf1}) and (\ref{rel1}). We proceed by induction on $\ell$. When $\ell=2$ the statement is trivial. So to prove the inductive step assume the statement holds for all $2\leq m\leq \ell-1$ and 
we will show it for $\ell$. Notice that since by assumption we have $k_1 \gamma_{n,1}k_2\gamma_{n,2 }k_3 \gamma_{n,3}\cdots k_\ell\gamma_{n,\ell}=e$, for all $n$, then there exists a smallest integer $2\leq d_1\leq \ell$ such 
that the sets $\{k_1\g_{n,1} \,:\,n\in \mathbb N \}$, $\{k_1\g_{n,1}k_{2}\g_{n,2} \,:\,n\in \mathbb N \}$, $\ldots$ ,$\{k_1\g_{n,1}k_{2}\g_{n,2}\cdots k_{d_1-1}\g_{n, d_1-1} \,:\,n\in \mathbb N \}$ are infinite while 
$\{k_1\g_{n,1}k_{2}\g_{n,2}\cdots k_{d}\g_{n, d_1} \,:\,n\in \mathbb N \}$ is finite. If $d_1=\ell$ then (\ref{inf1}) and (\ref{rel1}) follow trivially. 

If $d_1\leq \ell-1$ there exists an infinite sequence $(a_n)_n$ of integers and $c\in \G$ such that  $k_1\g_{a_n,1}k_{2}\g_{a_n,2}\cdots k_{d_1}\g_{a_n, d_1}=c$, for all $n$, and hence we have that 
$k_{2}\g_{a_n,2}\cdots k_{d_1}\g_{a_n, d_1} k''_1\g_{a_n,1}=e$, for all $n$, where we denoted by $k''_1=c^{-1}k_1$. In this case  (\ref{inf1}) and (\ref{rel1}) follow from the induction assumption.
 \end{proof}
 
 We note that the previous corollary can be easily generalized to the case of quasi-cocycles into representations $\pi$ which are mixing with relatively to a family of subgroups $\Cal G$. The statement is virtually the same 
 with the mention that, rather than a finite set, in this case $K$ will be a set which is small with respect to $\Cal G$ and equation (\ref{72}) will hold with $K$ replaced by $\cup^{[\ell/2]+1}_{j=1} K^j$. The proof is also very 
 similar to the one presented above and we leave it to the reader.
\begin{remark} The above results are an approximate translation to the quasi-cocycle perspective of Proposition 2.8 in \cite{DGO} which  states that if $\La< \G$ hyperbolically embedded, then $\La$ is \emph{almost malnormal} 
in $\G$ -- i.e., $\abs{\La \cap \g \La \g^{-1}}<\infty$ whenever $\g\not\in \La$.  Indeed Theorem 4.2 in \cite{HO} seems to suggest (though it is not explicitly proven) that starting from 
$0\in QZ^1_{as}(\La, \la_{\La}^{\oplus\infty})$, one can construct a quasi-cocycle $q\in QZ^1_{as}(\G,\la_{\G}^{\oplus\infty})$ such that $q|_\La \equiv 0$ and which is proper with respect to $\La$. By essentially the same 
techniques as above this, with some work, ought to imply the almost malnormality of $\La$.
\end{remark}

We make one final remark before passing to the main results of the paper. The following result is well known in particular cases, such as when the group has positive first $\ell^2$-Betti number. 

\begin{prop}Let $\G$ be a countable group which admit a non-proper, unbounded quasi-cocycle into a mixing orthogonal representation of $\G$. Then the following hold 
\begin{enumerate}\item $\G$ does not admit an infinite abelian normal subgroup; 
\item $\G\neq \G_1 \vee \G_2$ for $\G_i$ infinite such that $[\G_1,\G_2]=e$ and $\G_1$ is normal in $\G$.
\end{enumerate}  \end{prop}


\section{Central Sequences, Asymptotic Relative Commutants, and Inner Amenability}

In this section we prove the main results of this note. First we establish a fairly general theorem which classifies all the central sequences in von Neumann algebras arising from groups (or actions of groups) which admit 
unbounded quasi-cocycles into (relative) mixing representations.

\subsection{Non-Gamma factors} 

\begin{thm}\label{controlcentralseq} Let $\Gamma$ be a countable discrete group together with a family of subgroups $\mathcal G$ such that $\G$ satisfies condition $\mathsf{NC}(\Cal G)$. Let $(A,\tau)$ be any amenable von 
Neumann algebra equipped with a faithful, normal trace $\tau$, and let $\Gamma \curvearrowright (A,\tau)$ be any trace preserving action. Also assume that $\omega$ is a free ultrafilter on the positive integers $\mathbb N$.  

Then for any  asymptotically central sequence $(x_n)_n\in M' \cap M^{\omega}$ there exists a finite subset $\mathcal F\subseteq \mathcal G$ such that $(x_n)_n\in \vee_{\Sigma \in \mathcal F}(A \rtimes \Sigma)^{\omega} \vee M$ (i.e. the von Neumann subalgebra of $M^\omega$ generated by $M$ and $(A\rtimes \Sigma)^\omega$ for $\Sigma \in \mathcal F$). \end{thm}

\begin{proof} Let $q$ be an arbitrary (anti-)symmetric array into an orthogonal representation $\pi$ which is non-amenable, and consider the corresponding deformation  $V_{t}$ as defined in Section \ref{sec:deformations}. We fix the notation $A\rtimes \G=M\subset \tilde M=(L^\infty(X^\pi)\bar \otimes A)\rtimes \G$. We will prove first that $V_{t}$ converges to the identity on any element of $M' \cap M^{\omega}$. 

So, let $(x_n)_n \in M' \cap M^{\omega}$ and fix $\ve>0$.  
Since the representation $\pi$ is non-amenable, then by Proposition \ref{spectralgap} so is the Koopman representation $\sigma^\pi: \G\ra \mathcal U(L^2(X^\pi)\ominus \mathbb C1)$. This implies that one can find a finite subset $K\subset \G$ and $L>0$ such that for all $\xi\in  L^{2}(\tilde{M}) \ominus L^{2}(M)$ we have  
\begin{equation}\label{511'}\sum_{k\in K}\|u_k\xi-\xi u_k\|_2\geq L\|\xi\|_2.\end{equation}
By Proposition \ref{almostbimodular} above choose $t_\ve>0$ such that  for all $k\in K$ and all $t_\ve\geq t\geq 0$ we have 
\begin{equation}\label{511} \sup_n\|V_t(u_kx_nu_k^*)-u_kV_t(x_n)u_k^*\|_2 \leq \frac{\varepsilon L}{\sqrt 2|K|}.\end{equation}

\noindent Using the transversality property (Proposition \ref{transversality}) together with (\ref{511'}), (\ref{511}), and $\lim_n \|[u_k,x_n]\|_2=0$ for all $k\in K$ we see that for all $t_\ve\geq t\geq 0$ we have 
\begin{equation}\label{511''}\begin{split}&\limsup_n\| V_t(x_n)-x_n \|_2 \\ 
&\leq \limsup_n \sqrt2 \|e^\perp_MV_t(x_n) \|_2 \\
&\leq\limsup_n \frac{\sqrt2}{L} \sum_{k \in K}\| u_k  e^\perp_MV_t(x_n)u^*_k- e^\perp_MV_t(x_n) \|_{2} \\
&\leq \limsup_n\frac{\sqrt2}{L} \left(\sum_{k \in K} \| e_M^\perp V_t(u_k x_{n}u_k^*-x_n) \|_2 + \sum_{k \in K} \| V_t(u_k x_nu_k^*) - u_k V_t(x_{n})u_k^*\|_2\right)\\
&\leq\limsup_n\frac{\sqrt2}{L} \left(\sum_{k \in K} \| [u_k ,x_n] \|_2 +\frac{\ve L}{\sqrt 2}\right)\\
&=\ve,\end{split}\end{equation} 

\noindent which proves the desired claim.\\

From this point the proof breaks in to two cases:\\

\noindent {\bf Case 1.} {\it Let $\pi$ be an orthogonal representation which is non-amenable and mixing with respect to $\Cal G$. We further assume that $\pi$ admits an unbounded quasi-cocycle, and let $q\in QZ^1_{as}(\G,\pi)$ be any such one of some defect $D\geq 0$.}\\

In this case we show that the uniform convergence of $V_t$ will be sufficient to locate the asymptotic central sequences in $M$. Let $(x_n)_n\in M'\cap M^\omega$, and let $\varepsilon > 0$. From the first part there exists $t_{\varepsilon} > 0$ such that for all $0 \leq |t| \leq t_{\varepsilon}$ we have 
\begin{equation}\label{501} \limsup_{n} \| x_{n} - V_{t}(x_{n}) \|_{2} \leq \ve.
\end{equation}
For every $R\geq 0$ denote by $B_{R}= \lbrace g \in \Gamma \,:\, \|q(g) \| \leq R \rbrace$, and by $P_{R}$ the orthogonal projection from $L^{2}(M)$ onto $\overline{span} \lbrace au_{g} \,:\, a \in A, g \in B_{R} \rbrace$. Then (\ref{501}) together with a simple computation show that there exists $C\geq 0$ for which 
\begin{equation}\label{503} \limsup_{n} \| x_{n} - P_{C}(x_{n}) \|_{2} \leq \ve.  \end{equation}
Using the triangle inequality, for every $y \in \mathcal{U}(M)$ we have that
\begin{equation*} \label{502}\begin{split} \| yP_{C}(x_{n}) - P_{C}(x_{n})y \|_{2} &\leq \| y(P_{C}(x_{n}) - x_{n}) - (P_{C}(x_{n}) - x_{n})y \|_{2} + \| [y, x_{n}] \|_{2}\\& \leq 2 \| P_{C}(x_{n}) - x_{n} \|_{2} + \| [y, x_{n}] \|_{2} \end{split}
\end{equation*}
and when this is further combined with (\ref{503}) we get 
\begin{equation}\label{504} \limsup_{n} \| yP_{C}(x_{n}) - P_{C}(x_{n})y \|_{2} \leq 2 \ve, \text{ for all }y \in \mathcal U(M). 
\end{equation}

By Corollary \ref{quasi-control}, for $\g \in (B_{2C+2D})^{c} =\G\setminus B_{2C+2D}$ in there exist finite subsets $F, K \subset \Gamma$, $\mathcal F \subset \mathcal{G}$ such that 
\begin{equation}\label{506}\g(B_{C} \setminus F \mathcal{F} K) \cap (B_{C} \setminus F \mathcal F K)\g = \emptyset.\end{equation} 
Next we show that $(x_n)_n\in  \vee_{\Sigma \in \mathcal F}(A \rtimes \Sigma)^{\omega} \vee M$. Suppose by contradiction this is not the case. Thus by subtracting from $(x_n)_n$ its conditional expectation onto 
$\vee_{\Sigma \in \mathcal F}(A \rtimes \Sigma)^{\omega} \vee M$ we can assume that 
$0\neq (x_{n}) \perp \vee_{\Sigma \in \mathcal F}(A \rtimes \Sigma)^{\omega} \vee M$. Since  $K,F \subset \Gamma$, $\mathcal{F} \subset \mathcal{G}$ are finite this further implies that
\begin{equation} \label{505}\lim_n\| P_{F \mathcal{F} K}(x_{n}) \|_{2} = 0.
\end{equation} 

\noindent Picking  $y=u_{\g}$ with $\g \in (B_{2C+2D})^{c}$ in (\ref{504}), we obtain that 
\begin{equation*} \limsup_{n} \| u_{\g}P_{C}(x_{n}) - P_{C}(x_{n})u_{\g} \|_{2} \leq 2 \ve, 
\end{equation*} and using this in combination with (\ref{505}) we have 
\begin{equation} \label{507}\begin{split} &\limsup_{n} \| u_{\g}P_{B_{C} \setminus F \mathcal{F} K}(x_{n}) - P_{B_{C} \setminus F \mathcal{F} K}(x_{n})u_{\g} \|_{2} \\ 
&\leq  \limsup_{n} \| u_{\g}P_{B_{C}}(x_{n}) - P_{B_{C}}(x_{n})u_{\g} \|_{2} + 2\limsup_{n} \| P_{F \mathcal{F} K}(x_{n}) \|_{2}  \\ &\leq 2 \ve. \end{split}
\end{equation}
Altogether, relations (\ref{507}), (\ref{506}) and (\ref{505}) lead to the following inequality 
\begin{equation*}\begin{split} 4 \ve^{2} &\geq \limsup_{n} \| u_{\g}P_{B_{C} \setminus F \mathcal{F} K}(x_{n}) - P_{B_{C} \setminus F \mathcal{F} K}(x_{n})u_{\g} \|^{2}_{2} \\
& = \limsup_{n}\left ( \| u_{\g}P_{B_{C} \setminus F \mathcal{F} K}(x_{n}) \|^{2}_{2} + \| P_{B_{C} \setminus F \mathcal{F} K}(x_{n})u_{\g} \|^{2}_{2}\right) \\ 
&=2 \limsup_{n}  \|  P_{B_{C} \setminus F \mathcal{F} K}(x_{n}) \|^{2}_{2} \\
&= 2\limsup_{n}\left (\| P_{B_{C}}(x_{n}) \|^{2}_{2} - \| P_{F \mathcal{F} K}(x_{n}) \|^{2}_{2}\right) \\
&= \limsup_{n} 2 \| P_{B_{C}}(x_{n}) \|^{2}_{2} \\
&= 2(1- \ve)^2,\end{split}
\end{equation*}
which for $\ve$ small enough is a contradiction. $\blacksquare$ \\

\noindent{\bf Case 2.} {\it Let $\pi$ be a non-amenable representation, and let $q$ be an array associated to $\pi$ which is proper with respect to $\Cal G$.}\\

For this case, we will suppose by contradiction that $(x_n)_n\not\in \vee_{\Sigma \in \mathcal F}(A \rtimes \Sigma)^{\omega} \vee M$. Let $\xi_n := e_M^\perp V_t(x_n)$. By essentially the same argument as in Theorem 3.2 of 
\cite{CS}, there would then exist a constant $c>0$ such that $\nor{\xi_n}_2\geq c$. Define a sequence of states $\vp_n$ on $\bb B(L^2(X^\pi))$ by $\vp_n(T) := \frac{1}{\nor{\xi_n}_2}\ip{(T\otimes 1)\xi_n}{\xi_n}$. Using 
proposition \ref{almostbimodular}, for every $\g\in\G$ it follows that $\abs{\vp_n(\pi_\g T \pi_\g^*) - \vp_n(T)}\to 0$ uniformly for $\nor{T}_\infty\leq 1$. Thus, by taking an arbitrary ultralimit $\vp(T) := \Lim_n \vp_n(T)$, 
we have constructed a state $\vp$ on $\bb B(L^2(X^\pi))$ which is invariant under conjugation by the unitaries $\pi_\g$, $\g\in\G$, contradicting the assumption that $\pi$ is a non-amenable representation.

\end{proof}

A well-known theorem of Connes (Theorem 2.1 in \cite{Connes-inj}) shows that property Gamma is equivalent to the existence of a net $\xi_n\in L^2(M)\ominus \bb C\hat 1$ of unit-norm vectors such that 
$\nor{x\xi_n - \xi_n x}_2\to 0$ for all $x\in M$. 

\begin{cor}If in the previous theorem we assume in addition that the family $\mathcal G$ consists only of non-inner amenable subgroups then $M'\cap M^\omega \subseteq A^\omega\rtimes \G$. Therefore if the actions 
$\G\ca A$ is also free and strongly ergodic then $M$ does not  have property \emph{Gamma} of Murray and von Neumann.\end{cor}



\subsection{Non-inner amenability}

Similar techniques can can be applied in the realm of groups to provide a fairly large class of groups which are not inner amenable; in particular, this extends some of the results covered already by Theorem \ref{controlcentralseq}. A group $\G$ is called \emph{inner amenable} if there exists a finite additive measure $\mu$ on the subsets  $\G\setminus \{e\}$ of total mass one such that  $\mu(X)=\mu(\g^{-1}X\g)$ for all $X\subset \G\setminus \{e\}$. However, in order to apply the method described in the previous theorem we need to use an alternative, $C^*$-algebraic  characterization of inner amenability; $\G$ is inner amenable if and only if there exists a sequence of unit vectors $(\xi_n)_n\in \ell^2 (\G) \ominus \mathbb Ce$ such that $\lim_n\|x\xi_n - \xi_n x\|_2 = 0$, for all $x \in C^*_r(\G)$. 
To properly state our result we need to introduce the following definition. 

\begin{definition}\label{def-icc} Let $\G$ be a countable group and let $\Sigma <\G$ be a subgroup. We say that $\G$ is \emph{i.c.c.\ over $\Sigma$} if for every finite subset $F\subset \G$ there exists $\g\in  \G$ such that 
$\g \left(F\Sigma F\right) \g^{-1} \cap F\Sigma F=\{e\}$.\end{definition} 

We used this terminology only because it naturally extends the classical i.c.c.\ notion for groups. The next result is probably folklore but we include a proof for the sake of completeness.

\begin{prop} \label{icc}Any countable group $\G$ is i.c.c.\ if and only if $\G$ is i.c.c.\ over $\Sigma=\{e\}$.\end{prop}

\begin{proof} One can easily see that the reverse implication holds, so we need only show the forward implication. This follows immediately once we show that for every finite subset 
$ K\subset \G\setminus \{e\}$ there exists $\g\in \G$ such that $\g K\g^{-1}\cap  K =\emptyset$. We proceed by contradiction, so suppose there exists a finite set $K_0\subset \G\setminus \{e\}$ such that for 
all $\g\in \G$ we have that \begin{equation}\label{16}\g K_0\g^{-1}\cap  K_0 \neq \emptyset.\end{equation} Consider the Hilbert space $\mathcal H =\ell^2(\G\setminus \{e\})$ and denote by $\xi$ the characteristic
function of $K_0$. Since $K_0$ is finite then $\xi \in \mathcal H$. From (\ref{16}), a simple calculation shows that \begin{equation}\label{17}\langle u_\g\xi u_\g^{-1}, \xi\rangle \geq \frac{1}{|K_0|}>0\text{, for all }
\g\in \G.\end{equation} Therefore if we denote by $\mathcal K \subset \mathcal H$ the closed, convex hull of the set $\{u_\g\xi u_{\g^{-1}} \,:\, \g\in \G\}$ and denote by $\zeta$ the unique $\|\cdot \|_2$-minimal 
element in $\mathcal K$. Then from (\ref{17}) we have that $\langle \zeta, \xi\rangle \geq 1/|K_0|>0$, in particular $\zeta \neq 0$. Hence, if we decompose $\zeta=\sum_{\g\in \G\setminus \{e\}}\zeta_\g\delta_\g$, 
there exists $\la\in \G\setminus \{e\}$ such that $\zeta_\la\neq 0$. On the other hand, by uniqueness, $\zeta$ satisfies that $u_\g\zeta u_{\g^{-1}}=\zeta$, for all $\g\in \G$; hence, we have $0\neq \zeta_\la=\zeta_{\g\la \g^{-1}}$,
for all $\g\in \G$. Since $\zeta \in \ell^2(\G)$ it follows that the orbit under conjugation $\{\g\la\g^{-1}\,:\,\g\in \G\}$ is finite thus contradicting the i.c.c.\ assumption on $\G$.  \end{proof}
Now with these notations at hand we are ready to state the main theorem.

\begin{thm}\label{non-inner} Let $ \G$ be an i.c.c., countable, discrete group together with a family of subgroups $\mathcal G$ so that $\G$ is i.c.c.\ over every subgroup $\Sigma\in \Cal G$. If $\G$ satisfies condition 
$\mathsf{NC}(\Cal G)$, then $\G$ is not inner amenable. \end{thm}

 Since any multiple of the left-regular representation of any non-amenable group is both non-amenable and mixing, the theorem shows that every non-amenable i.c.c.\ group $\G$ satisfying 
 $QH_{as}^1(\G , \la_\G^{\oplus\infty} )\neq \emptyset$  is not inner amenable; in particular, $L\G$ does not have property \emph{Gamma} of Murray and von Neumann. By Theorem 1.4  in \cite{HO}  this covers all non-amenable 
 groups which admit hyperbolically embedded subgroups, so our result recovers Theorem 8.2 (f) from \cite{DGO}. The result also demonstrates that all groups with positive first $\ell^2$-Betti number are either finite or non-inner 
 amenable since for any such group $\G$ it holds that $H^1(\G, \ell^2(\G))\neq 0$, cf. \cite{PeTho}. Finally, we point out that in the case that $\G$ has a non-amenable orthogonal representation $\pi$ which admits a proper symmetric array (i.e., $\G$ belongs to the class $\Cal{QH}$), then the fact that $\G$ is not inner amenable is already contained in Proposition 1.7.5 in \cite{CS}.

\begin{proof} [Proof of Theorem \ref{non-inner}]We will proceed by contradiction, so suppose $\G$ is inner amenable. Thus there exists a sequence $(\xi_n)_n\in \ell^2(\G)\ominus \mathbb C e$ of unit vectors such that  for 
every $x\in C^*_r(\G)$ we have 
 \begin{equation}\label{ninner1}\lim_{n\ra 0}\|x\xi_n-\xi_nx\|_2=0.\end{equation} 
 
Let $\pi$ be a non-amenable representation, and let $q$ be any (anti-)symmetric array associated to $\pi$. As in the previous theorem, let $M=L\G$ and let $\tilde M =L^\infty(Y^\pi)\rtimes \G$ be the Gaussian construction 
associated with $\pi$. Consider $V_t:L^2(M)\ra L^2(\tilde M)$ for $t\in \mathbb R$, the associated path of unitaries as defined in subsection \ref{sec:deformations}. Using the non-amenability of $\pi$, the same spectral 
gap argument as in Theorem \ref{controlcentralseq} shows that $\lim_{t\ra 0}\left(\sup_n\|e^\perp_M V_t(\xi_n)\|_2\right)=0$. By the transversality property (Proposition \ref{transversality})  this gives that   
$\lim_{t\ra 0}\left(\sup_n\|\xi_n-V_t(\xi_n)\|_2\right)=0$. Then a simple calculation shows that for every $\ve >0$ there exists $C\geq 0$ such that 
\begin{equation}\label{ninner2}\sup_n\|\xi_n- P_{B'_C}(\xi_n)\|_2\leq \ve.\end{equation} As before, we have denoted by $P_{B'_C}$ the orthogonal projection from $\ell^2(\G)$ onto the Hilbert subspace $\ell^2(B'_C)$ with 
$B'_C=\{
\la \,:\, \|r(\la)\|\leq C, \la\neq e\}$ being the ball of radius $C$ centered and pierced at the identity element $e$. 
 
Using the triangle inequality,  relations (\ref{ninner1}) and (\ref{ninner2}) show that for every $\g\in \G$ we have 
\begin{equation}\label{ninner3} \limsup_n \|u_\g P_{B'_C}(\xi_n)-P_{B'_C}(\xi_n)u_\g \|_2\leq 2\ve.\end{equation}

If we write $\xi_n=\sum_{\eta \in \G}  \xi^n_\eta\delta_\eta$ with $\xi^n_\eta\in \mathbb C$ then (\ref{ninner3}) gives the following estimates 

\begin{equation*}\label{ninner4}\begin{split}4\ve^2&\geq \limsup_n\|\sum_{\eta\in B'_C}\xi^n_\eta\delta_{\g\eta}-\xi^n_\eta\delta_{\eta\g }\|^2_2\\
&=\limsup_n\left(\|\sum_{s\in \g B'_C\setminus B'_C\g}\xi^n_{\g^{-1}s}\delta_s\|^2+\|\sum_{s\in B'_C\g\setminus \g B'_C}\xi^n_{s\g^{-1}}\delta_s\|^2+\|\sum_{s\in \g B'_C\cap B'_C\g}(\xi^n_{\g^{-1}s}-\xi^n_{s\g^{-1}})\delta_s\|^2\right)\\
&=\limsup_n\left(2\sum_{s\in B'_C}|\xi^n_{s}|^2+\sum_{s\in \g B'_C\cap B_C\g}|\xi^n_{\g^{-1}s}-\xi^n_{s\g^{-1}}|^2-\sum_{s\in \g B'_C\cap B'_C\g}(|\xi^n_{\g^{-1}s}|^2+|\xi^n_{s\g^{-1}}|^2)\right).
\end{split}\end{equation*}
Since $\xi_n$ are unital vectors then the previous estimate together with (\ref{ninner2}) show that 
\begin{equation*}\label{ninner5}\begin{split}&4\ve^2+ \limsup_n\sum_{s\in \g B'_C\cap B'_C\g}\left (|\xi^n_{\g^{-1}s}|^2+|\xi^n_{s\g^{-1}}|^2\right )\\
&\geq \limsup_n\left(2\|P_{B'_C}(\xi_n)\|^2_2+\sum_{s\in \g B'_C\cap B'_C\g}|\xi^n_{\g^{-1}s}-\xi^n_{s\g^{-1}}|^2\right)\\
&\geq 2(1-\ve^2). 
\end{split}\end{equation*} 
Altogether, the previous inequalities imply that for every $\g\in \G$ we have 
\begin{equation}\label{ninner7}\limsup_n \left (\sum_{s\in (\g B'_C\g^{-1}\cup \g^{-1}B'_C\g)\cap B'_C}|\xi^n_s|^2\right)\geq 2(1-3\ve^2),\end{equation} 
and since $\sum_s|\xi^n_s|^2=\|\xi_n\|^2=1$ we conclude that,  for every $\g\in \G$ we have 
\begin{equation}\label{ninner11}\limsup_n\|P_{A_\g}(\xi_n)\|^2_2=\limsup_n \left (\sum_{s\in A_\g}|\xi^n_s|^2\right)\geq 1-6\ve^2.\end{equation} 
\noindent Here for every $\g\in \G$ we have denoted by $A_\g=\g B'_C\g^{-1}\cap B'_C$ and for a set $\Omega \subset \G$ we denote by $P_{\Omega}$ the orthogonal projection from $\ell^2(\G)$ onto $\ell^2(\Omega)$.\\

\noindent{\bf Claim.} We have assumed that $\G$ is i.c.c.\ over every $\Sigma \in \mathcal G$ and admits a map $q: \G\to \Cal H$ such that either: (1) $q$ is a proper array associated to $\pi$; or (2) $\pi$ is mixing 
with respect to $\mathcal G$  and $q$ is an unbounded quasi-cocycle. In either case we claim that there exist finite subsets $\mathcal G_o \subset \mathcal G$ and $F\subset \G$ such that 
$A_\g\subseteq \cup_{\Sigma \in \mathcal G_o}F\Sigma F$.  When $q$ is a proper array this follows because $A_\g \subset B_C$ which by the properness assumption is contained in a finite union of finitely many left-right 
translates of groups in $\mathcal G$. In the other case our claim follows from Corollary \ref{quasi-control}. $\blacksquare$ \\

Using our claim, after passing to a subsequence of $\xi_n$, the inequality (\ref{ninner11}) implies that for all $n$ we have 
\begin{equation*}\|P_{\cup_{\Sigma\in \mathcal G_o}F\Sigma F}(\xi_n)\|^2_2=\sum_{s \in \cup_{\Sigma \in \mathcal G_o}F \Sigma F}|\xi^{n}_s|^2  \geq 1-6\ve^2.\end{equation*}
Thus since $\mathcal G_o$ is finite then passing one more time to a subsequence of $\xi_n$ there exists $\Sigma \in \mathcal G_o$ such  that for all $n$ we have 
\begin{equation}\label{12}\|P_{F\Sigma F\setminus \{e\}}(\xi_n)\|^2_2=\|P_{F\Sigma F}(\xi_n)\|^2_2=\sum_{s\in F\Sigma F}|\xi^n_s|^2\geq \frac{1-6\ve^2}{|\mathcal G_o|}:=D_\ve>0,\end{equation} 
where $P_{F\Sigma F}$ denotes the orthogonal projection from $\ell^2(\G)$ onto $\ell^2(F\Sigma F)$. 
Next we claim that from  the assumption that $\G$ is i.c.c. over every group in $\mathcal G$  one can construct inductively two infinite sequences $(F_\ell)_\ell, (G_\ell)_\ell$ of finite subsets of $\G $ such that 
$(F_\ell\Sigma G_\ell)\setminus \{e\}$ are pairwise disjoint sets and  $\limsup_n\|P_{F_\ell\Sigma G_\ell \setminus \{e\}}(\xi_{n})\|^2_2\geq D_\ve$, for all $\ell\in \mathbb{N}$. Therefore using Parseval's identity for every 
$\ell \in \mathbb N$ we have that $1= \limsup_n \|\xi_{n}\|^2_2\geq \limsup_n \sum^\ell_{i=1}\|P_{F_\ell\Sigma G_\ell \setminus \{e\}}(\xi_{n}) \|^2_2\geq \ell D_\ve$, which is a contradiction when $\ell$ is arbitrarily large; 
hence it will follow that $\G$ is not inner amenable.

In the remaining part of the proof of this case we show the claim above by induction on $\ell$. Since the case $\ell=1$ follows immediately from (\ref{12}) by letting $F_1=G_1=F$, we only need to show the induction step. 
So assume that for $1\leq i \leq \ell$ we have constructed finite subsets $F_i, G_i\subset \G$ such that sets $(F_i\Sigma G_i)\setminus \{e\}$ are pairwise disjoint  and
 \begin{equation}\label{14}\limsup_n \|P_{F_i\Sigma G_i\setminus \{e\}}(\xi_{n})\|^2_2\geq D_\ve\text{, for all }1\leq i\leq \ell.\end{equation} 
 
 Now we will indicate how to build the  subsets $F_{\ell+1}, G_{\ell+1}\subset \G$ with the required properties.   Since $F_i$ and $G_i$ are finite sets then so are $F'=\cup^\ell_{i=1} F_i$  and $G'=\cup^\ell_{i=1} G_i$ and 
 from the assumption there exists $\mu\in \G$ such that $\mu(F'\Sigma G')\mu^{-1} \cap  F'\Sigma G'=\{e\}$.

 Using the projection formula $P_{\g \Omega \g^{-1}}(\xi)=u_\g P_\Omega( u_{\g^{-1}}\xi u_{\g})u_{\g^{-1}}$ for $\g \in \G$, $\xi\in \ell^2(\G)$, and $\Omega \subseteq \G$ in combinations with the triangle inequality, 
 $\lim_n\|u_{\mu^{-1}} \xi_n u_{\mu}-\xi_n\|_2=0$, and (\ref{14}) we see that  
 \begin{equation*}\label{13}\begin{split}\limsup_n\|P_{\mu (F'\Sigma G')\mu^{-1}\setminus \{e\}}(\xi_{n})\|_2&=\limsup_n\|P_{\mu(F'\Sigma G')\mu^{-1}}(\xi_{n})\|_2\\&\geq \limsup_n\left (\|u_\mu P_{F'\Sigma G'}(\xi_{n})u_{\mu^{-1}}\|_2- \|P_{\mu(F'\Sigma G')\mu^{-1}}(\xi_{n}) - u_\mu P_{F'\Sigma G'}(\xi_n)u_{\mu^{-1}}\|_2\right )\\
 &= \limsup_n\left(\|P_{F'\Sigma G'}(\xi_{n})\|_2- \|P_{F'\Sigma G'}( u_\mu \xi_{n}u_\mu -\xi_{n})\|_2\right) \\
 &\geq  \limsup_n\left(\|P_{F'\Sigma G'}(\xi_{n})\|_2- \| u_{\mu^{-1}} \xi_{n}u_{\mu} -\xi_{n}\|_2 \right)\\ 
  &\geq \limsup_n  \|P_{F'\Sigma G'}(\xi_{n})\|_2-\lim_n  \| u_{\mu^{-1}}\xi_{n}u_{\mu} -\xi_{n}\|_2 \\ 
&\geq D_\ve. 
\end{split}\end{equation*}
Altogether, this computation and the choice of $x$ show that the sets $F_{\ell+1}=\mu F'$ and $G_{\ell+1}=G'\mu^{-1}$ satisfy the required conditions.
\end{proof}

As a corollary we recover and generalize a result of de la Harpe and Skandalis.

\begin{prop}[de la Harpe and Skandalis \cite{HarSka}] (1) If $\G$ is a lattice in a real, connected, semi-simple Lie group $G$ with trivial center and no compact factors, then $\G$ is not inner amenable. (2) In general, let $G = G_1\times G_2$ be a unimodular l.c.s.c.\ group such that $G_1$ is topologically i.c.c.\footnote{That is, for any compact neighborhood $e\in K\subset G$ of the identity, and any neighborhood $U\ni e$, there exists $g_1,\dotsc, g_n\in G$ such that $g_1 K g_1^{-1} \cap \dotsb \cap g_n K g_n^{-1}\subset U$.} and has property (HH). Then any i.c.c.\ irreducible lattice $\G< G$ is not inner amenable.
\end{prop}

\begin{proof} Note that in case (1) $\G$ is i.c.c.\ as a consequence of Borel density, cf.\ \cite{Harpe-hyp}. Without loss of generality we may assume $\G$ is irreducible, as any lattice is a product of such. In the case (1), we may further assume that $G$ does not have property (T); otherwise, this would imply that $\G$ is an i.c.c.\ property (T) group, therefore not inner amenable. Hence, $G$ has a factor with property (HH), cf.\ \cite{OPII}. Since $G$ is without compact factors, it is topologically i.c.c.; thus, we have reduced case (1) to case (2).

So, now assume we are in the general situation of case (2). By Example \ref{ex-lattice} and Theorem \ref{non-inner}, we need only show that $\G$ is i.c.c. relative to any subgroup $\Sigma$ such that the projection into $G_1$ is pre-compact. This is true by the topological i.c.c.\ property and the irreducibility which implies that the image of $\G$ under the projection is dense in $G_1$.
\end{proof}


\subsection{Further consequences}

If $M$ is a separable $\rm II_1$ factor, a trivial consequence of $M$ not having property $Gamma$ is that $M$ cannot be written as an infinite tensor product of non-scalar finite factors. However, whether not having property $Gamma$ implies the ``stabilized'' version of this property seems to be unresolved in the literature; that is to say, whether $M$ not having property $Gamma$ implies $M\bar\otimes R$ is not isomorphic to a infinite tensor product of non-amenable factors. As usual, $R$ denotes the hyperfinite $\rm II_1$ factor. Under the stronger assumption that each factor in the tensor decomposition is non-$Gamma$, the answer is negative by Theorem 4.1 in \cite{Po2}. In the more restrictive case of factors of groups satisfying $\mathsf{NC}$, we now show that the strong negative does obtain. We remark in passing that Ozawa and Popa's unique prime decomposition theorem (Theorem 2 in \cite{OP}, cf.\ Theorem C in \cite{CS}) demonstrates that much stronger structural results can be obtained by placing more 
stringent assumptions on the groups.

\begin{thm}\label{stable-non-gamma} Let $\{\G_j\,:\, 1\leq j\leq n\}$ be a finite collection of non-amenable i.c.c.\ groups satisfying condition $\mathsf{NC}$. 
If $\{N_i\,:\,  i\in I\}$  is any countable collection of non-amenable ${\rm II}_1$ factors such that $ L\G_1\bar \otimes L\G_2\bar\otimes \cdots \bar \otimes L\G_n\bar\otimes R \cong \bar\otimes_{i\in I}N_i $ then 
$I$ is a finite set. 
\end{thm}


\begin{proof} Suppose by contradiction that $|I|=\infty$; hence, there exists an infinite sequence $I_n\subset I$ of finite subsets such that $I_n\subset I_{n+1}$ and 
$\cup_nI_n=I$. Denote by $J_n=I^c_n$, by $N(J_n)=\bar\otimes_{i\in J_n} N_i$, and by $N(I_n)=\bar\otimes_{i\in I_n} N_i$. Also denote by $M= L\G_1\bar \otimes L\G_2\bar\otimes \cdots \bar \otimes L\G_n\bar\otimes R$ and 
fix $\omega$ a free ultrafilter on $\mathbb N$. Applying the spectral gap argument as in the proof of Theorem \ref{controlcentralseq} we will show that there exists $s\in \mathbb N$ 
such that $N(J_{s})^\omega\subseteq  L\G_1\bar\otimes \cdots \bar\otimes L\G_n\bar\otimes R^\omega $.   

Throughout the proof we will view $M$ as $R\bar\otimes  L\G$ where $\G :=\G_1\times \cdots \times \G_n$. Since by Theorem \ref{non-inner} each $\G_j$ is non-inner amenable, so is $\G$. Thus by Proposition \ref{spectralgap} there 
exist a finite subset $K\subset \G$ and $C>0$ such that for all $x\in M$ we have
\begin{equation}\label{511'''}\sum_{\g\in K} \|xu_\g-u_\g x\|_2^2\geq C\|E_R(x)-x\|^2_2.
\end{equation} 

Since $M$ is the inductive limit of $N(I_n)$ as $n\ra\infty$ and $N(I_n)$ commute with $N(J_n)$ for all $n$ then using (\ref{511'''}) together with some basic approximations and the triangle inequality we obtain 
the following: for every $\varepsilon>0$ there exist $s_\varepsilon\in\mathbb N$ such that for every $x\in (N(J_{s_\varepsilon}))_1$ we have $\|E_R(x)-x\|_2\leq \varepsilon$. If we let $\varepsilon$ to be small enough,
then applying Popa's intertwining techniques from \cite{P} we obtain that a corner of $N(J_{s_\varepsilon})$ intertwines into $R$ inside $M$. This however is a contradiction because no corner of a non-amenable factor can be intertwined into an amenable von Neumann algebra. Therefore $I$ cannot be infinite, and we are done.  \end{proof}

Following \cite{Sa} a factor $M$ is called \emph{asymptotically abelian} if there exists a sequence of automorphisms $\theta_n\subset {\rm Aut}(M)$ such that $\lim_n\|\theta_n(x)y-y\theta_n(x) \|_2=0$ for all $x,y\in M$. Next we will show that using the previous techniques one can provide a fairly class of algebras that are not asymptotically abelian. In particular the result provides many new examples of McDuff factors which are not asymptotically abelian, enlarging the class of examples found in \cite{Sa}.

\begin{prop} Let $\{\G_i\,:\, 1\leq i\leq n\}$ be a finite collection of non-amenable i.c.c.\ groups satisfying condition $\mathsf{NC}$. If $R$ is any amenable finite factor, 
then the factor $ L\G_1\bar \otimes L\G_2\bar\otimes \cdots \bar \otimes L\G_n\bar\otimes R $ is not asymptotically abelian.
\end{prop}

\begin{proof} For simplicity we denote by $P :=L\G_1\bar \otimes L\G_2\bar\otimes \cdots \bar \otimes L\G_n$ and $M := P\bar\otimes R $. Suppose by contradiction that $M$ is asymptotically abelian; thus,
there exists a sequence of automorphisms $\theta_k \in {\rm Aut}(M)$ such that $\|\theta_k(x)y-y\theta_k(x)\|_2\ra 0$ as $k\ra \infty$, for all $x,y\in M$. This means that $(\theta_k(x))_k\in M'\cap M^\omega$,  
for all $x\in M$. Since by Theorem \ref{non-inner} all $\G_i$ are not inner amenable then so is $\G=\G_1\times \cdots \times \G_n$ and by Proposition \ref{spectralgap} we have  $(\theta_k(x))_k\in  R^\omega  $ for all $x\in M$ and all $1\leq i\leq n$. 
 This implies that of every $(y_k)_k \in P^\omega$ we have that $\lim_k\| \theta_k(x)y_k-y_k\theta_k(x)\|_2
=0$, for all $x\in M$. Since the automorphisms $\theta_n$ are $\tau$-invariant we obtain that $\lim_k\|x\Phi_k(y_k)-\Phi_k(y_k)x\|_2=0$, for all $x\in M$, where $\Phi_k=\theta_k^{-1}$. Thus $\Phi_k(y_k)$ 
is an asymptotically central sequence, so by the same argument as before we have that $(\Phi_k(y_k))_k\in  R^\omega$, for all $(y_k)_k \in P^\omega$. Thus for all $(y_k)_k \in P^\omega$ we have 
\begin{equation}\label{7101}\lim_{k\ra \omega}\|E_{ R}(\Phi_k(y_k))-\Phi_k(y_k)\|_2=0.\end{equation}  

Next we show that this is will lead to a contradiction. Indeed since $\G_i$ are assumed non-amenable i.c.c.\ groups it follows that $P$ is a non-amenable ${\rm II}_1$ factor. Since $ R$ is amenable then for each 
$k\in \mathbb N$ no corner of $\Phi_k(P)$ can be intertwined in the sense of Popa into $ R$ inside $Q$, \cite{P}. Thus by Theorem 2.3 in \cite{P} for each $k\in \mathbb N$ there exists a unitary $u_k\in \mathcal U(P)$ such 
that $\|E_{ R}(\Phi_k(u_k))\|_2\leq 1/k$. Since $u_k$ is a unitary then using this in combination with (\ref{7101}) we have that
\begin{multline*} 1=\lim_k\|\Phi_k(u_k)\|_2\leq  \lim_k\left(\|E_{ R}(\Phi_k(u_k))\|_2 +\|\Phi_k(u_k)-E_{ R}(\Phi_k(u_k))\|_2\right)\\ \leq \lim_k\left(1/k +\|\Phi_k(u_k)-E_{ R}(\Phi_k(u_k))\|_2\right) =0,\end{multline*} which is a contradiction.\end{proof}


\subsection{Remark on central sequences and simplicity of group $\rm C^*$-algebras}\label{simplicity} In \cite{Tucker-Drob}, the author poses the question of whether the reduced $\rm C^*$-algebra of an i.c.c., countable, discrete group $\G$ is simple if the group has positive first $\ell^2$-Betti number. In light of the fact that, by Theorem 8.12 in \cite{DGO}, $C^*$-simplicity is known for groups admitting a non-degenerate hyperbolically embedded subgroup, we propose that this ought to be true in a more general context:

\begin{question} If $\G$ is an i.c.c., countable, discrete group such satisfying condition $\mathsf{NC}$, is $C_r^*(\G)$ simple?
\end{question}

 In this remark, we will recall by the work of Akemann and Pedersen \cite{AkePed} how $\rm C^*$-simplicity of an i.c.c.\ group is equivalent to the non-existence of certain central sequences in $C_r^*(\G)$. Let $\G$ be an i.c.c., 
 countable, discrete group. Recall, a \emph{central sequence} in $C_r^*(\G)$ is a bounded sequence $(z_n)$ such that $\nor{xz_n - z_n x}_\infty\to 0$ for all $x\in C_r^*(\G)$. A central sequence is said to be trivial if there 
 exists a sequence of scalars $(c_n)$ such that $\nor{z_n - c_n 1}_\infty\to 0$. A central sequence is $(z_n)$ is said to be \emph{summable} if each $z_n\geq 0$ and $\sum_n z_n = z\in L\G$, where the sum is understood to converge in the strong topology.

We remark that even though $L(\bb F_2)$ does not have property $Gamma$, $C_r^*(\bb F_2)$ is rife with non-trivial central sequences by Theorem 2.4 in \cite{AkePed} due to the topology of its spectrum being highly non-Hausdorff. However, $C_r^*(\bb F_2)$ has no non-trivial \emph{summable} central sequences as a consequence of Theorems 3.1 and 3.3 in \cite{AkePed}.

\begin{prop}[Akemann and Pedersen] If $\G$ is an i.c.c., countable, discrete group, then $C_r^*(\G)$ is simple if and only if every summable central sequence is trivial.
\end{prop}

\noindent It would be highly interesting to investigate whether any $C^*$-algebraic ``rigidity'' techniques can be developed which, similarly to the von Neumann algebraic techniques used above, could be used to rule out the presence of summable (norm) central sequences. Specifically, it would be interesting to know whether ``spectral gap'' type phenomena exist at the $C^*$-level. Take the following concrete situation: $\G$ is a non-amenable countable, discrete group,  $\G\ca X$ is a Bernoulli action, and $B := L^\infty(X)\rtimes_{r} \G$ is the reduced crossed product. Suppose that $x_n$ is a positive, summable sequence in $B$ such that $\nor{a x_n - x_n a}\to 0$ for all $a\in C_r^*(\G)$. Is it the case that there exists a sequence $x_n'\in C_r^*(\G)\subset B$ so that $\nor{x_n - x_n'}\to 0$?

\section*{Acknowledgements} We would like to thank Denis Osin and Jesse Peterson for useful conversations on the subject of this paper.

\bibliographystyle{amsplain}

\begin{thebibliography}{10}

\bibitem{AkePed} C. A. Akemann and G. K. Pedersen, {\it Central sequences and inner derivations of separable $C^*$-algebras}, Amer. J. Math. {\bf 101} (1979), 1047--1061.


\bibitem{BBF} M. Bestvina, K. Bromberg, and K. Fujiwara, {\it Bounded cohomology via quasi-trees}, preprint (2013). \nolinkurl{arXiv:1306.1542(v1)}

\bibitem{BrOz} N. P. Brown and N. Ozawa, \textit{${\rm C}^\ast$-algebras and finite-dimensional approximations}, Graduate Studies in Mathematics, vol. 88, AMS, Providence, RI.

\bibitem{CIfree} I. Chifan and A. Ioana, \textit{Strong solidity in tensor products of free group factors}, unpublished manuscript (2009).

\bibitem{CP} I. Chifan and J. Peterson, \textit{Some unique group-measure space decomposition results}, Duke Math. J. {\bf 162} (2013), 1--44.

\bibitem{CS} I. Chifan and T. Sinclair, \textit{On the structural theory of $\,{\rm II}_1$ factors of negatively curved groups}, Ann. Sci. \'{E}c. Norm. Sup\'{e}r. {\bf 46} (2013), 1--33.

\bibitem{CSU} I. Chifan, T. Sinclair, and B. Udrea, \textit{On the structural theory of $\,{\rm II}_1$ factors of negatively curved groups, II: Actions by product groups}, Adv. Math. {\bf  245} (2013), 208--236.

\bibitem{Connes-inj} A. Connes, {\it Classification of injective factors. Cases II$_1$, II$_\infty$, III$_\lambda$, $\lambda\not=1$}, Ann. of Math. {\bf 104} (1976), 73--115.

\bibitem{DGO}F. Dahmani, F. Guirardel, and D. Osin, \textit{Hyperbolically embedded subgroups and rotating families in groups acting on hyperbolic spaces}, preprint (2011). \nolinkurl{arXiv:1111.7048(v3)}

\bibitem{Effros} E. G. Effros, {\it Property $\Gamma$ and inner amenability}, Proc. Amer. Math. Soc. {\bf 47} (1975), 483--486.

\bibitem{Ham-jems} U. Hamenst\"adt, {\it Bounded cohomology and isometry groups of hyperbolic spaces}, J. Eur. Math. Soc. {\bf 10} (2008), 315--349.

\bibitem{Harpe-hyp} P. de la Harpe, {\it Groupes hyperboliques, alg\`ebres d'op\'erateurs et un th\'eor\`eme de Jolissaint}, C. R. Acad. Sci. Paris S\'er. I Math. {\bf 307} (1988), 771--774.

\bibitem{Harpe-survey} P. de la Harpe, {\it On simplicity of reduced $C^*$-algebras of groups}, Bull. London Math. Soc. {\bf 39} (2007), 1-26.

\bibitem{HarSka} P. de la Harpe and G. Skandalis, {\it Les r\'eseaux dans les groupes semi-simples ne sont pas int\'erieurement moyennables}, Enseign. Math. {\bf 40} (1994), 291--311.

\bibitem{Hou} C. Houdayer, {\it Structure of II$_1$ factors arising from free Bogoljubov actions of arbitrary groups}, preprint (2012). \nolinkurl{arXiv:1209.5209(v2)}

\bibitem{HO} M. Hull and D. Osin, \textit{Groups with hyperbolically embedded subgroups}, Alg. \& Geom. Topology, to appear. \nolinkurl{arXiv:1203.5436(v4)}



\bibitem{McD} D. McDuff, \textit{Central sequences and the hyperfinite factor}, Proc. London Math. Soc. {\bf 21} (1970), 443--461.

\bibitem{Min} I. Mineyev, \textit{Straightening and bounded cohomology of hyperbolic groups}, Geom. Funct. Anal. {\bf 11} (2001), 807--839.

\bibitem{MMS} I. Mineyev, N. Monod, and Y. Shalom, \textit{Ideal bicombings for hyperbolic groups and applications}, Topology {\bf 43} (2004), 1319--1344.

\bibitem{MinYa} I. Mineyev and A. Yaman, \textit{ Relative hyperbolicity and bounded cohomology}, preprint (2009).




\bibitem{Murray-vN-4} F. J. Murray and J. von Neumann, {\it On rings of operators. IV}, Ann. of Math. {\bf 44} (1943), 716--808.

\bibitem{Osin-ac} D. Osin, {\it Acylindrically hyperbolic groups}, preprint (2013). \nolinkurl{arXiv:1304.1246(v1)}

\bibitem{OzSolid} N. Ozawa, \textit{Solid von Neumann algebras}, Acta Math. {\bf 192} (2004), 111--117.

\bibitem{OP} N. Ozawa and S. Popa, \textit{Some prime factorization results for type II$_1$ factors}, Invent. Math. {\bf 156} (2004), 223--234.

\bibitem{OPI} N. Ozawa and S. Popa, {\it On a class of II$_1$ factors with at most one Cartan subalgebra}, Ann. of Math. {\bf 172} (2010), 713--749.

\bibitem{OPII} N. Ozawa and S. Popa, {\it On a class of II$_1$ factors with at most one Cartan subalgebra, II}, Amer. J. Math. {\bf 132} (2010), 841--866.

\bibitem{PetL2} J. Peterson, \textit{$L^2$-rigidity in von {N}eumann algebras}, Invent. Math. {\bf 175} (2009), 417--433.

\bibitem{PeSi} J. Peterson and T. Sinclair, \textit{On cocycle superrigidity for Gaussian actions}, Ergodic Theory Dynam. Systems {\bf 32} (2012), 249--272.

\bibitem{PeTho} J. Peterson and A. Thom, {\it Group cocycles and the ring of affiliated operators}, Invent. Math. {\bf 185} (2011), 561--592.



\bibitem{P} S. Popa, {\it Strong Rigidity of ${\rm II}_1$ Factors Arising from Malleable Actions of $w$-Rigid Groups I}, Invent. Math. {\bf 165} (2006), 369--408.

\bibitem{PoICM} S. Popa, \textit{Deformation and rigidity for group actions and von Neumann algebras}, International Congress of Mathematicians, Vol. I, 445--477, Eur. Math. Soc., Z\"urich, 2007. 

\bibitem{PoFree} S. Popa, \textit{On Ozawa's property for free group factors}, Int. Math. Res. Not. (2007),  Art. ID rnm036, 10pp.

\bibitem{Po2} S. Popa, {\it On spectral gap rigidity and Connes invariant $\chi(M)$}, Proc. Amer. Math. Soc. {\bf 138} (2010), 3531--3539.

\bibitem{PoVa-hyp} S. Popa and S. Vaes, {\it Unique Cartan decomposition for II$_1$ factors arising from arbitrary actions of hyperbolic groups}, J. Reine Angew. Math., to appear.


\bibitem{Sa} S. Sakai, \textit{Asymptotically abelian ${\rm II}_1$ factors}, Publ. RIMS, Kyoto Univ. Ser. A Vol. 4 (1968), 299--307.

\bibitem{Sin} T. Sinclair, \textit{Strong solidity of group factors from lattices in SO(n,1) and SU(n,1)}, J. Funct. Anal. {\bf 260} (2011), 3209--3221.
  
\bibitem{Sisto} A. Sisto, {\it Contracting elements and random walks}, preprint (2011). \nolinkurl{arXiv:1112.2666(v1)}

\bibitem{Tho} A. Thom, \textit{Low degree bounded cohomology invariants and negatively curved groups}, Groups Geom. Dyn. {\bf 3} (2009), 343--358.

\bibitem{Tucker-Drob} R. D. Tucker-Drob, {\it Shift-minimal groups, fixed price 1, and the unique trace property}, preprint (2012).

\bibitem{Vae} S. Vaes, {\it An inner amenable group whose von Neumann algebra does not have property Gamma}, Acta Math. {\bf 208} (2012), 389--394.
\end{thebibliography}

\end{document}